\documentclass[a4paper,12pt]{article}
\usepackage{latexsym}
\usepackage{amsmath}
\usepackage{amssymb}
\usepackage{enumerate}

\usepackage{color}

\usepackage{theorem}
\newtheorem{theorem}{Theorem}[section]
\newtheorem{proposition}[theorem]{Proposition}
\newtheorem{lemma}[theorem]{Lemma}

\newtheorem{MT}{Main Theorem}

\theorembodyfont{\rmfamily}
\newtheorem{proof}{\textmd{\textit{Proof.}}}

\newtheorem{remark}[theorem]{Remark}

\makeatletter

\@addtoreset{equation}{section}
\makeatother

\newcommand{\qedd}{\hfill \Box}

\newcommand{\wt}{\widetilde}

\title{The Structure Theorem for The cut locus of 
a Certain Class of Cylinders of Revolution II
\footnote{
Mathematics Subject Classification (2010)\,: 53C22.}
.}
\author{Pakkinee CHITSAKUL}
\date{}
\begin{document}

\maketitle

%%%%%%%%%%%%%%%%%%%%%%%%%%%%%%%%%%%%%%%%%%%%%%%%%%%%%%%%%%%%%%%%%%%%%%%%%%%%%
\begin{abstract}

In the previous paper \cite{C}, the structure of the cut locus was 
determined for a class of surfaces of revolution homeomorphic to a cylinder. 
In this paper, we prove the structure theorem of the cut locus for 
a wider class of surfaces of revolution homeomorphic to a cylinder.

\end{abstract}
%%%%%%%%%%%%%%%%%%%%%%%%%%%%%%%%%%%%%%%%%%%%%%%%%%%%%%%%%%%%%%%%%%%%%%%%%%%%%

%%%Introduction%%%%%%%%%%%%%%%%%%%%%%%%%%%%%%%%%%%%%%%%%%%%%%%%%%%%%%%%%%%%%%
\section{Introduction}
The following structure theorem was proved in \cite{C} for a class of surfaces of 
revolution homeomorphic to a cylinder.

%%%%%%%%%%%%%%%%%%%%%%%%%%%%%%%%%%%%%%%%%%%%%%%%%%%%%%%%%%%%%%%%%%%%%%%%%%%%%%
{\noindent\bf Theorem} 
Let $(M ,ds^2)$ be a complete Riemannian manifold $R^1\times S^1$ 
with a warped product metric $ds^2=dt^2+m(t)^2 d\theta^2$ 
of the real line $(R^1,dt^2)$ and the unit circle $(S^1,d\theta ^2)$. 
Suppose that the warping function $m$ is a positive-valued even function 
and the Gaussian curvature of $M$ is decreasing along the half meridian 
$t^{-1}[0,\infty] \cap \theta^{-1}(0)$. If the Gaussian curvature of $M$ 
is positive on $t=0$, then the structure of the cut locus $C_q$ of a 
point 
$q \in  \theta^{-1}(0)$ in $M$ is given as follows:
\begin{enumerate}
\item
The cut locus $C_q$ is the union of a subarc of the parallel $t=-t(q)$ 
opposite to $q$ and the meridian opposite to 
$q$ if $|t(q)|<t_0 := \sup \{ t>0|m'(t)<0 \} $ and $\varphi(m(t(q)))<\pi$. 
More precisely,
\begin{equation*}
C_q=\theta^{-1}(\pi) \cup (t^{-1}(-t(q)) \cap 
\theta^{-1}[\varphi(m(t(q))),2\pi-\varphi(m(t(q)))])
\end{equation*}
\item
The cut locus $C_q$ is the meridian $\theta^{-1}(\pi)$ opposite to $q$ 
if $\varphi(m(t(q))) \geq \pi$ or if $|t(q)|\geq t_0$.
\end{enumerate}
Here, the {\it half period  function } $\varphi(\nu)$ on $(\inf m, m(0))$ is defined as 
\begin{equation}\label{eq1.1}
\varphi{(\nu)}:= 2 \int^{0}_{-\xi(\nu)} \frac{\nu}{m\sqrt{m^2-\nu^2}}dt
= 2 \int_{0}^{\xi(\nu)} \frac{\nu}{m\sqrt{m^2-\nu^2}}dt,
\end{equation}
where $\xi(\nu):= \min \{t>0 | m(t)= \nu \}$.
Notice that the point $q$ is an arbitrarily given point 
if the coordinates $(t,\theta)$ are chosen so as to satisfy $\theta(q)=0$.
\bigskip

Crucial properties of the manifold  $(M ,ds^2)$ in the theorem above are
\begin{enumerate}
\item
$M$ has a  reflective symmetry with respect to a parallel.
\item
 The Gaussian curvature is decreasing along each upper half meridian.
\end{enumerate}

In this paper, the second property is replaced by the following property:

 {\it The cut locus of a point on $\tilde t=0$ is a nonempty subset of $\tilde t=0,$}
for the universal covering space $(\tilde{M},d\tilde{t}^2+m(\tilde{t})^2d\tilde{\theta}^2)$ of a cylinder of revolution $(M,dt^2+m(t)^2d\theta^2)$ with a reflective symmetry with respect to the parallel $t=0.$
\bigskip

We will prove the following structure theorem of the cut locus for a cylinder of revolution satisfying the property above.

%%%%%%%%%%%%%%Main Theorem%%%%%%%%%%%%%%%%%%%%%%%%%%%%%%%%%%%%%%%%%%%%%%%%%%%%
\begin{MT}
Let $(M,ds^2)$ denote a complete Riemannian manifold $R^1 \times S^1$ 
with a warped product metric $ds^2=dt^2+m(t)^2d\theta^2$ of the real line 
$(R^1,dt^2)$ and the unit circle $(S^1,d\theta^2)$, and by 
$(\widetilde{M},d\tilde{t}^2+m(\tilde{t})^2d\tilde{\theta}^2)$
we denote the universal covering space of $(M,ds^2)$. Suppose that $m$ is
an even positive-valued function. If the cut locus of a point on 
$\tilde t^{-1}(0)$ is a nonempty subset of $\tilde{t}^{-1}(0)$, 
then the cut locus $C_q$ of a point $q$ of $M$
with $|t(q)|<t_0 := \sup \{t>0\;| \;\:  m'(t)<0\}$
equals the union of a subarc of the parallel $t=-t(q)$
opposite to $q$ and the meridian opposite to $q$. More precisely,
there exists a number $t_\pi\in[0,t_0)$ such that
for any point $q$ with $|t(q)|<t_\pi,$
$$C_q=\theta^{-1}(\pi) \cup \left( t^{-1}(-t(q)) \cap 
\theta^{-1}[\varphi(m(t(q))),2\pi-\varphi(m(t(q)))]\right)$$
and for any point $q$ with $t_\pi\leq |t(q)|<t_0,$
$C_q=\theta^{-1}(\pi).$
Moreover,
if $t_0$ is finite, then $C_q=\theta^{-1}(\pi)$ for any point $q$ with $|t(q)|=t_0.$

Here the coordinates $(t,\theta)$ are chosen so as to satisfy $\theta(q)=0$.
Notice that the domain of the half period   function $\varphi(\nu)$ is  $(m(t_0),m(0))$ (respectively $(\inf m,m(0))$ ) if $t_0$ is finite (respectively infinite).
\end{MT}
%%%%%%%%%%%%%%%%%%%%%%%%%remark%%%%%%%%%%%%%%%%%%%%%%%%%%%%%%%%%%%%%%%%%%%%%%%
\begin{remark}
If the Gaussian curvature of the manifold $M$ in the Main Theorem is nonpositive on 
$\tilde t^{-1}(t_0,\infty)$, then the cut locus $C_q$
 of any point $q$ with $|t(q)|>t_0$ is equal to  $\theta^{-1}(\pi),$ the meridian opposite 
 to $q$.
 \end{remark}

We refer to \cite{C}, \cite{SST}  and \cite{ST} for some fundamental properties of geodesics on a surface of revolution and the structure theorem of the cut locus on a surface.
 %%%%%%%%%%%%%Section A necessary and sufficient condition%%%%%%%%%%%%%%%%%%%%%
\section{A necessary and sufficient condition 
for $\mathbf{\varphi(\nu)}$ to be decreasing}

A complete Riemannian manifold $(M,ds^2)$ homeomorphic to $R^1\times S^1$ is called a
{ \it cylinder of revolution}
 if $ds^2=dt^2+m(t)^2d \theta^2$  is  a warped 
product metric of the real line $(R^1,dt^2)$ and 
the unit circle $(S^1,d\theta^2)$.

Throughout this paper, we assume that the warping function $m$ of a cylinder of revolution $M$ is an even function. 
Hence $M$ has a reflective symmetry with respect to $t=0$, 
which is called the {\it{equator.} }
Let $(\wt{M},d\tilde{s}^2)$ denote the universal covering space of 
$(M,ds^2)$. Thus $d\tilde s^2=d\tilde t^2+m(\tilde t)^2d\tilde\theta^2.$
Since $m'(0)=0,$ it follows from Lemma 7.1.4 in \cite{SST} that the equator $t=0$ and $\tilde t=0$ are geodesics in $M$ and $\wt M$ respectively.

The following  lemma is a corresponding one to Lemma 3.2 in \cite{BCST} in the 
case of a two-sphere  of revolution.
%%%%%%%%%%%%%%%%%%%%%Lemma A1%%%%%%%%%%%%%%%%%%%%%%%%%%%%%%%%%%%%%%%%%%%%%%%
\begin{lemma}\label{lem2.1}
If the cut locus of a point in $\tilde{t}^{-1}(0)$ 
is a nonempty subset of $\tilde{t}^{-1}(0)$, 
then the Gaussian curvature of $\wt M$ is positive on 
$\tilde{t}^{-1}(0)$ and for any $t>0$ satisfying $m'|_{(0,t)}<0$, 
the function $\varphi(\nu)$ is decreasing on $(m(t),m(0))$.
\end{lemma}
\begin{proof}
Let $q$ be  an end point of the cut locus of  a point $p\in\tilde  t^{-1}(0).$ 
Since the end point $q$ is conjugate to $p$ along the subarc of $\tilde t^{-1}(0),$ 
the Gaussian curvature on $\tilde t^{-1}(0)$ is positive.
We omit the proof of the second claim, since the proof of Proposition 4.6 in \cite {C} is applicable.$\qedd$
\end{proof}

%%%%%%%%%%%%%%%%%%%%%Lemma A2%%%%%%%%%%%%%%%%%%%%%%%%%%%%%%%%%%%%%%%%%%%%%%%
\begin{lemma}\label{lem2.2}
Suppose that the Gaussian curvature of $\wt M$ 
is positive on $\tilde{t}^{-1}(0)$. Let $t>0$ be any  number 
satisfying $m'|_{(0,t)}<0$. If $\varphi(\nu)$ is  decreasing 
on $(m(t),m(0))$ then for any point $\tilde{p} \in \tilde{t}^{-1}(0)$,
$C_{\tilde{p}} \cap \tilde{t}^{-1}(-t,t)$ is a nonempty subset of 
$\tilde{t}^{-1}(0)$. Here $C_{\tilde p}$ denotes the cut locus of $\tilde p.$
\end{lemma}
\begin{proof}
Choose an arbitrary point
 $\tilde p\in\tilde  t^{-1}(0)$ 
and fix it.
Since the Gaussian curvature is positive constant on
 $\tilde t=0,$ there exists a conjugate point of $\tilde p$ along the subarc of $\tilde t=0.$
Thus, $C_{\tilde p}\cap \tilde t^{-1}(-t,t)$ is nonempty. We omit the proof of the claim that $C_{\tilde p}\cap \tilde t^{-1}(-t,t)$ 
is a subset of $\tilde t^{-1}(0),$ since 
the proof of Lemma 3.3 in \cite {BCST} is still valid in our case.
$\qedd$
\end{proof}

Combining Lemmas \ref{lem2.1} and \ref{lem2.2} we get

%%%%%%%%%%%%%%%%%Proposition A3%%%%%%%%%%%%%%%%%%%%%%%%%%%%%%%%%%%%%%%%%%%%%
\begin{proposition}\label{prop2.3}
Suppose that  $m' \neq 0$ on $(0,\infty)$. Then the cut locus of a point 
on $\tilde{t}^{-1}(0)$ is a nonempty subset of $\tilde{t}^{-1}(0)$ 
if and only if the Gaussian curvature of ${\wt M}$ is positive on 
$\tilde{t}^{-1}(0)$ and the half period function $\varphi(\nu)$ defined by \eqref{eq1.1} is decreasing on 
$(\inf m, m(0))$.
\end{proposition}
%%%%%%%%%%%%%%%%%%%%%%%%%%%%%%%%%%%
\section{Preliminaries}
From now on, we assume that {\it the cut locus of a point on $\tilde t=0$ is a nonempty subset of $\tilde t=0.$}  Hence, from Lemma \ref{lem2.1}, the function $\varphi(\nu)$ is decreasing on $(m(t_0),m(0)),$ where 
$t_0:=\sup\{\: t>0\: |  \:  m'(t)<0\}$ and   $m(t_0)$ means $\inf m$ when $t_0=\infty.$
For each $\nu\in[0,m(0))$ let $\gamma_{\nu}:[0,\infty)\to \wt M$ denote a unit speed geodesic emanating from the  point $\tilde p:=(\tilde t,\tilde\theta)^{-1}(0,0)$ on $\tilde t^{-1}(0)$ with Clairaut constant $\nu.$ 
It is known (see \cite {C}, for example) that $\gamma_\nu$ intersects $\tilde t^{-1}(0)$ again  at the point $(\tilde t,\tilde \theta)^{-1}(0,\varphi(\nu))$ if $\nu$ is greater than $m(t_0).$ 
Notice that $\gamma_\nu$ is a submeridian of $\tilde\theta=0,$ when $\nu=0.$

\begin{lemma}\label{lem3.1}
If $0\leq\nu\leq m(t_0),$ then $\gamma_{\nu}$  is not tangent to any parallel arc $\tilde t=c.$
In particular, the geodesic does not intersect $\tilde t=0$ again.

\end{lemma}
\begin{proof}
Since there does not exist a cut point of $\tilde p$ in $\tilde t\ne 0,$ the subarc $\gamma_\nu|_{[0,l(\nu)]}$ of $\gamma_\nu$ is minimal for each $\nu\in(m(t_0),m(0)).$
Here $l(\nu)$ denotes the length of the subarc of $\gamma_\nu$  having end points $\tilde p$ and $(\tilde t,\tilde \theta)^{-1}(0,\varphi(\nu)).$ Therefore, the limit geodesic $\gamma_{m(t_0)}=\lim_{\nu\searrow m(t_0)} \gamma_\nu|_{[0,l(\nu)]}$ is a ray emanating from $\tilde p$ and in particular, $\gamma_{m(t_0)}$ is not tangent to any parallel arc and does not intersect $\tilde t=0$ again.
We will prove that for any $\nu\in[0,m(t_0)),$ $\gamma_\nu$ is not tangent to any parallel arc. Suppose that for some $\nu_0\in(0,m(t_0)),$ $\gamma_{\nu_0}$ is tangent to a parallel arc. Since $\wt M$ has a reflection symmetry with respect to $\tilde t=0,$
we may assume that $(\tilde t\circ \gamma_{\nu_0})'(0)<0$ and $(\tilde t \circ\gamma_{m(t_0)})'(0)<0.$
By applying the Clairaut relation at the point $\tilde p,$ $\gamma_{\nu_0}|_{(0,t)}$ lies in the domain $D$ cut off by $\gamma_{m(t_0)}$ and the submeridian $\gamma_0$ of $\tilde \theta=0$ for some positive $t.$ Since there does not exist a cut point of $\tilde p$ in $\tilde t^{-1}(-\infty,0),$ the geodesic $\gamma_{\nu_0}$ does not intersect $\gamma_{m(t_0)}$ again. Hence $\gamma_{\nu_0}|_{(0,\infty)}$ lies in the domain $D.$
Since $\gamma_{\nu_0}$ is tangent to a parallel arc, the geodesic intersects $\tilde t=0$ again, which is a contradiction.

$\qedd$
\end{proof}

\begin{lemma}\label{lem3.2}
Let $\tilde\gamma_\nu : R\to \wt M$ denote a unit speed geodesic with Clairaut constant $\nu\in(0,m(t_0)].$
If $\tilde\gamma_\nu$ passes through a  point of  $\tilde t^{-1}(-t_0,t_0),$ then $\tilde\gamma_\nu$ is not tangent to any parallel arc $\tilde t=c.$

\end{lemma}
\begin{proof}
First, 
we will prove that $\tilde\gamma_\nu$ intersects $\tilde t=0$ for any $\nu\in[0,m(t_0)].$
Supposing that $\tilde \gamma_\nu$ does not intersect $\tilde t=0$ for some $\nu\in[0,m(t_0)],$ we will get a contradiction.
Since $\wt M$ has a reflective symmetry with respect to $\tilde t=0,$
we may assume that $(\tilde t\circ\tilde\gamma_\nu)(s)<0$ for any real number $s.$
By the Clairaut relation, $(\tilde t\circ\tilde \gamma_\nu)'(s)\ne 0$ for any $s$ satisfying $-t_0<\tilde t\circ\tilde\gamma_\nu(s)<0<t_0.$
From the assumptions, we may assume that $\tilde t\circ\tilde\gamma_\nu(0)\in(-t_0,0).$
If  $(\tilde t\circ\tilde\gamma_\nu)'(0)>0$ (respectively  $(\tilde t\circ\tilde\gamma_\nu)'(0)<0$) then 
 $\tilde t\circ\tilde \gamma_\nu(s)$ is increasing  (respectively decreasing) and bounded above by $0. $ Thus,
there exists a unique limit $ -t_0<\tilde t_1:=\lim_{s\to\infty}\tilde t\circ\tilde\gamma_\nu(s)\leq 0$ (respectively $ -t_0<\tilde t_1:=\lim_{s\to-\infty}\tilde t\circ\tilde\gamma_\nu(s)\leq 0$).
It follows from Lemma 7.1.7 in \cite{SST} that $m'(\tilde t_1)=0$ and $m(\tilde t_1)=\nu.$
This is a contradiction, since $\nu\in[0,m(t_0)]$ and $-t_0<\tilde t_1\leq 0.$
 Therefore, $\tilde\gamma_\nu$ intersects $\tilde t=0$ for any $\nu\in[0,m(t_0)],$ and hence by Lemma \ref{lem3.1}, the geodesic is not tangent to any parallel arc.

$\qedd$
\end{proof}

\begin{lemma}\label{lem3.3}
If $t_0=\sup\{\; t>0\; | \; m'(t)<0\}$ is finite, then
any subarc of the parallel arc $\tilde t=-t_0$ is minimal, i.e., the parallel arc is a straight line.
Hence, $\tilde t=t_0$ is also  a straight line.
\end {lemma}
\begin{proof}
Since $m'(t_0)=0,$ the parallel arc $\tilde t=-t_0$ is a geodesic by Lemma 7.1.4 in \cite{SST}.
Let $c$ be a geodesic emanating from  a point on $\tilde t=-t_0$ which is not tangent to $\tilde t=-t_0.$
By Lemma \ref{lem3.2},
$c$ is not tangent to any parallel arc.
In particular, $c$ does not intersect $\tilde t=-t_0$ again. This implies that $\tilde t=-t_0$ is a straight line. Since $\wt M$ has a reflective symmetry with respect to $\tilde t=0,$ $\tilde t=t_0$ is also a straight line.

$\qedd$
\end{proof}

%%%%%%%%%%%%%%%%%%%%%%%%%%%%%%%%%%%%%%%%%%%%%%%%%%%%%%%%%%%%%%%%%%%%%
%%%%The cut locus of a point on M%%%%%%%%%%%%%%%%%%%%%%%%%%%%%%%%%%%%$
\section{The cut locus of a point in $\wt M$}
 Choose any point $q$ in $\wt M$ with $-t_0<\tilde t(q)<0.$ 
Without loss of generality, we may assume 
 that $\tilde\theta(q)=0$. 
For each $\nu\in[0,m(0))$ let $\gamma_{\nu}:[0,\infty)\to \wt M$ denote a geodesic emanating from the  point $\tilde p:=(\tilde t,\tilde\theta)^{-1}(0,0)$ on $\tilde t^{-1}(0)$ with Clairaut constant $\nu.$ The geodesic $\gamma_\nu$ intersects $\tilde t=0$ again at the point $(\tilde t,\tilde\theta)^{-1}(0,\varphi(\nu)),$ if $\nu>m(t_0).$

 We consider two geodesics $\alpha_\nu$ and $\beta_\nu$ emanating 
 from the point $q=\alpha_\nu(0)=\beta_\nu(0)$ with  Clairaut constant $\nu>0$. 
 Here we assume that the angle
$ \angle (({\partial}/{\partial \tilde t})_q,\alpha'_\nu (0)) $
made by the tangent vectors $(\partial/\partial \tilde t)_q$ and $\alpha'_{\nu(0)}$ is grater than the angle 
 $\angle (({\partial}/{\partial\tilde  t})_q,\beta'_{\nu} (0))$ by $(\partial/\partial \tilde t)_q$ 
and 
$\beta'_{\nu}(0),$ if $\nu<m(t(q)).$ Notice that 
$\alpha_\nu=\beta_\nu$ if $\nu=m(t(q)).$

It follows from Lemma 5.1 in \cite{C} that $\alpha_\nu$ and $\beta_\nu$ intersect again at the point $(\tilde t,\tilde \theta)^{-1}(u,\varphi(\nu)),$ where  $u:=-\tilde t(q), $    if   $\nu\in (m(t_0),m(u)).$
The subarcs of $\alpha_\nu$ and $\beta_\nu$ having end points $q$ and $(\tilde t,\tilde\theta)^{-1}(u,\varphi(\nu))$ have the same length and 
its length  equals   $l(\nu),$ where $l(\nu)$ denotes the length the subarc  of $\gamma_\nu$ having end points $\tilde p$ and  $(\tilde t,\tilde\theta)^{-1}(0,\varphi(\nu)).$

%\end{document}

%%%%%%%%%Lemma5.1%%%%%%%%%%%%%%%%%%%%%%%%%%%%%%%%%%%%%%%%%%%%%%%%%%%%%%%%%
\begin{lemma}\label{lem4.1}
Let $q$ be  a point in $\wt M$ with $|\tilde t(q)|\in(0,t_0).$ 
Then, for any $\nu\in( m(t_0),m(u)],$ where $u=-\tilde t(q),$
$\alpha_\nu|_{[0,l(\nu)]}$ and $\beta_\nu|_{  [0,l(\nu)]}$ are minimal geodesic segments joining $q$ to the point $(\tilde t,\tilde \theta)^{-1}(u,\tilde \theta(q)+\varphi(\nu)), $ and 
in particular,
$\{(\tilde t,\tilde \theta) \: | \: \tilde t=u, \tilde \theta\geq \varphi(m(u))+\tilde\theta(q)\} $
is a subset of the cut locus of the point $q.$
\end{lemma}
\begin{proof}
Without loss of generality, we may assume that $\tilde\theta(q)=0$ and $-t_0<\tilde t(q)<0.$
We will prove that $\alpha_{\nu}|_{[0,l(\nu)]}$ is a minimal geodesic segment joining $q$ to the point $\alpha_{\nu}(l(\nu))=(\tilde t,\tilde\theta)^{-1}(u,\varphi(\nu)).$
Suppose that $\alpha_{\nu_0}|_{[0,l(\nu_0)]}$ is not minimal for some $\nu_0 \in (m(t_0),m(u)].$
Here we assume that $\nu_0$ is the minimum solution $\nu=\nu_0$ of 
$\varphi(\nu)=\varphi(\nu_0).$

Let  $\alpha:[0,d(q,x)]\to M$ be a unit speed minimal geodesic  segment  joining $q$ to $x:=\alpha_{\nu_0}(l(\nu_0))=(\tilde t,\tilde\theta)^{-1}(u,\varphi(\nu_0)).$
Hence,  $\varphi(\nu_1)=\varphi(\nu_0)=\tilde\theta(x)$ and 
$\alpha $ equals $\alpha_{\nu_1}|_{[0,l(\nu_1)]}$ or $\beta_{\nu_1}|_{[0,l(\nu_1)]},$
 where
$\nu_1 \in( m(t_0),m(u))$ 
denotes the Clairaut constant of $\alpha.$
By Lemma \ref{lem2.1}, $\varphi(\nu)=\varphi(\nu_0)$ for any
 $\nu\in[\nu_0,\nu_1].$
 Hence, by Lemma 3.2 in \cite{C}
 we get, $l(\nu_1)=l(\nu_0).$
This implies that $\alpha_{\nu_0}|_{[0,l(\nu_0)]}$ is minimal, which is a  contradiction, since we assumed that  $\alpha_{\nu_0}|_{[0,l(\nu_0)]}$ is not minimal.
Therefore,   
for any $\nu\in(m(t_0),m(u)],$
the  geodesic segments $\alpha_{\nu}|_{[0,l(\nu)]}$ and 
$\beta_{\nu}|_{[0,l(\nu)]}$
 are minimal geodesic segments joining $q$ to the point $(\tilde t,\tilde\theta)^{-1}(u,\varphi(\nu))=\alpha_\nu(l(\nu)).$
In particular,   the point 
$\alpha_\nu(l(\nu))=\beta_\nu(l(\nu))   $ 
is a cut point of $q.$
$\qedd$
\end{proof}

\begin{proposition}\label{prop4.2}
The cut locus of any point $q$ with $| \tilde t(q)|<t_0$ equals  the set
$$ 
\{(\tilde t,\tilde \theta) \: | \: \tilde t=u, \tilde \theta\geq |\varphi(m(u))|\}. $$
Here the coordinates $(\tilde t , \tilde\theta)$ are chosen so as to satisfy $\tilde\theta(q)=0.$
\end{proposition}
\begin{proof}
Without of loss of generality, we may assume that $-t_0<\tilde t(q)<0.$
By Lemma \ref{lem4.1}, the 
geodesic segments $\alpha_\nu|_{[0,l(\nu)]}$ and $\beta_\nu|_{[0,l(\nu)]}$
are minimal for any $\nu\in( m(t_0),m(u)].$
Hence their limit geodesics $\alpha^{-}:=\alpha_{ m(t_0)}$ and $\beta^+:=\beta_{ m(t_0)}$ are 
rays, that is, any  their  subarcs are minimal.
Since $\wt M$ has a reflective symmetry with respect to $\tilde\theta=0,$
it is trivial from Lemma \ref{lem4.1} that 
the set 
$\{(\tilde t,\tilde \theta) \: | \: \tilde t=u, \tilde \theta\geq |\varphi(m(u))|\}$
is  a subset of the cut locus of $q.$
Suppose that there exists a cut point $y\notin 
\{(\tilde t,\tilde \theta) \: | \: \tilde t=u, \tilde \theta\geq| \varphi(m(u))| \}.$
Without loss of generality, we may assume that $\tilde\theta(y)>0=\tilde\theta(q).$
Since the cut locus of $q$ has  a tree structure,
there exists an end point $x$ of the cut locus in the set  $\{(\tilde t,\tilde\theta)\: |\; \tilde\theta>0\}\setminus D(\beta^+,\alpha^-),$
where $D(\beta^+,\alpha^-)$ denotes the closure of the unbounded domain cut off by $\beta^+$ and $\alpha^-.$
Hence, $x$ is conjugate to $q$
along  any minimal geodesic segments $\gamma$ joining $q$ to $x.$ 
Since such a minimal geodesic $\gamma$ runs in the set $\{(\tilde t,\tilde\theta)\: |\; \tilde\theta>0\}\setminus D(\beta^+,\alpha^-),$
by applying the Clairaut relation at the point $q,$ we get that
the Clairaut constant of $\gamma$ is positive and less than $m(t_0).$
Notice that the geodesics  $\beta^+$ and $\alpha^-$ have the same Clairaut constant $m(t_0).$
It follows from  Lemma \ref{lem3.1} that the  geodesic $\gamma$  cannot be tangent to any parallel arc.
From Corollary  7.2.1 in \cite{SST},  $\gamma$ has no conjugate point of $q,$ which is a contradiction.
$\qedd$
\end{proof}

\begin{lemma}\label{lem4.3}
Let $q$ be a point in $\wt M$ with $|\tilde t(q)|=t_0.$
Then, the cut locus of $q$ is empty.

\end{lemma}
\begin{proof}

We may assume that $\tilde t(q)=-t_0,$ since $\wt M$ has a reflective symmetry with respect to $\tilde t=0.$
Supposing that there exists a cut point $x$ of $q,$ we will get a contradiction.
Since $\wt M$ is simply connected, the cut locus has an end point.
Hence, we may assume that the cut point $x$ is an end point of $C_q.$
Let $\gamma$ be a minimal geodesic segment joining $q$ to $x.$
Then,  $x$ is a conjugate point of $q$ along $\gamma,$ since $x$ is an end point of the cut locus.
From Lemma \ref{lem3.3}, $\tilde t(x)\ne -t_0.$
By applying the Clairaut relation, we obtain that the Clairaut constant of $\gamma$ is smaller than $m(t_0),$
and hence $\gamma$ is not tangent to any parallel arc by Lemma \ref{lem3.2}.
Therefore, by Corollary 7.2.1 in \cite{SST}, there does not exist a conjugate point of q along $\gamma,$ which is a contradiction.

$\qedd$
\end{proof}

\begin{lemma}\label{lem4.4}
Let $q$ be a point in $\wt M$ with $|\tilde t(q)|> t_0.$
If the Gaussian curvature of $\tilde M$ is nonpositive on $\tilde t^{-1}(-\infty,-t_0)\cup \tilde  t^{-1}(t_0,\infty),$
then the cut locus of  the point $q$  is empty. 
\end{lemma}
\begin{proof}
Suppose that the cut locus of  a point $q$ with $|\tilde t(q)|> t_0$ is nonempty.
Since $\wt M$ has a reflective symmetry with respect to $\tilde t=0,$ we may assume that $\tilde t(q)< -t_0.$ 
Supposing the existence of  a cut point of $q,$ we will  get a contradiction.
We may assume that there exists an end point $x$ of $C_q,$ since $\wt M$ is simply connected.
Let $\gamma:[0,d(q,x)]\to \wt M$ be a unit speed minimal geodesic joining $q$ to $x.$
If  $\tilde  t(\gamma(s))\leq -t_0$ for any $s\in(0,d(q,x)],$ then $\gamma$ has no conjugate point of $q,$ since the Gaussian curvature is nonpositive on $\tilde t^{-1}(-\infty,-t_0)\cup\tilde t^{-1}(t_0,\infty).$
This contradicts the fact that $x$ is an end point of $C_q.$ 
Thus we may assume that  $\tilde t(\gamma(s))>-t_0$ for some $s\in(0,d(q,x)].$
This implies that $\gamma$ passes through a point of  $\tilde t^{-1}(-t_0,t_0).$ 
It follows from the Clairaut relation that the Caliraut constant of $\gamma$ is smaller than $m(t_0).$ Hence, from Corollary  7.2.1 in \cite{SST} and Lemma \ref{lem3.2}, there  does not exist a conjugate point of $q$ along $\gamma,$  which is a contradiction.

$\qedd$
\end{proof}

\bigskip
\noindent
{\it Proof of Main Theorem}\medskip\\
Since the functions $m$ and $\varphi$ are decreasing on $[0,t_0)$ and $(m(t_0),m(0))$ respectively, the composite function $\varphi\circ m$ is increasing on $(0,t_0).$
It  is clear to see that 
$\lim_{t\nearrow t_0}\varphi\circ m(t)=\infty,$
since the minimal geodesic segment $\gamma_\nu|_{[0,l(\nu)  ]}$ converges to the ray $\gamma_{m(t_0)}$ as $\nu\searrow m(t_0).$
Let $t=t_\pi\in[0,t_0)$ be a solution of $\varphi\circ m(t)=\pi.$ Define $t_\pi=0$ if there is no solution. Hence, $\varphi\circ m(t)\geq \pi $ on $[t_\pi,t_0)$ and 
$\varphi\circ m(t)\leq \pi$ on $(0,t_\pi).$ Now the Main Theorem is clear from Proposition \ref{prop4.2} and Lemma \ref{lem4.3}.

%%%%%%%%%%%%Section Example%%%%%%%%%%%%%%%%%%%%%%%%%%%%%%%%%%%%%%%%%%%%%%%%%%
\section{ A family of cylinders of revolution}
An example 
  of a cylinder of revolution satisfying the two properties 1 and 2 in the introduction 
was given by Tamura \cite{T}.
 The Riemannian metric $ds^2$ is defined  by 
$ds^2=dt^2+e^{-t^2}d\theta^2.$
It is easy to see  that 
$m'=-2t \cdot m<0$
on $ (0,\infty),$
and the Gaussian curvature $G(q)$ at a point $q$ is
$-4t^2(q)+2.$
This implies that the Gaussian curvature is decreasing on each upper 
half meridian of the surface.
By Lemma 4.5 in \cite{C}, the cut locus of a point on $\tilde t=0$ on the universal covering space of the surface  is  a nonempty subset of $\tilde t=0.$ 
Hence, this surface satisfies the assumptions of the Main Theorem.
The following family of  cylinders of revolution shows that the converse is not true, i.e.,
 under the assumptions of the Main Theorem, the decline of the Gaussian curvature does not always hold.

In this section we give a family of cylinders of revolution $\{M_{\lambda}\}_{\lambda}:=\{(R\times S^1,dt^2+m_\lambda(t)^2d\theta^2)\}_{\lambda}$ 
satisfying
the assumtions in the Main Theorem, where $\lambda>1$ denotes a parameter and 
\begin{equation}\label{eq5.1}
m_{\lambda}(t):=\frac{\cosh t}{\sqrt{1+ \lambda \sinh^{2}t}}.
\end{equation}

\begin{lemma}\label{lem5.1}
The Gaussian curvature $G(q)$ at a point $q\in M_{\lambda}$ 
is given by 
\begin{equation} \label{eq5.2}
G(q)=(\lambda -1)(\frac{3}{h^2(t(q))}-\frac{2}{h(t(q))}),
\end{equation}
where $h(t)=1+ \lambda \sinh ^2 t$.
In particular, the Gaussian curvature  $G$ is not monotonic along the upper half meridian 
$\theta^{-1}(0) \cap t^{-1}(0,\infty)$.
\end{lemma}
\begin{proof}
From \eqref{eq5.1}, we get
\begin{equation}\label{eq5.3}
m'_\lambda(t)=(1-\lambda)m_\lambda(t)\tanh t/h(t),
\end{equation}
and
\begin{equation*}
m''_\lambda(t)=\left(   (1-\lambda){\tanh t}/{h(t)} \right)^2m_{\lambda}(t)+(1-\lambda)m_{\lambda}(t)(h(t)/\cosh^{2}t-h'(t)\tanh t)/h(t)^2.
\end{equation*}
Thus, we obtain
$$
m''_\lambda=(1-\lambda)m_\lambda(t)((1-\lambda)\tanh^2t+h(t)/\cosh^2t-h'(t)\tanh t)/h^2.$$
Since 
$(1-\lambda)\tanh^2t+h(t)/\cosh^{2}t=1$ holds,
we have 
\begin{equation*}
-{m''_\lambda(t)}/{m_\lambda(t)}=(\lambda-1)\left({3}/{h^2(t)}-{2}/{h(t)}\right).
\end{equation*}
Since $G(q)=-{m''_\lambda(t(q))}/{m_\lambda(t(q))},$ we
obtain
\eqref{eq5.2}.
By \eqref{eq5.4},
it is trivial to see that the Gaussian curvature is not monotonic along the upper half meridian.

$\qedd$
\end{proof}

\begin{lemma}\label{lem5.2}
Let $a,b\in(0,1)$ be numbers with $a<b.$
Then,
\begin{equation}\label{eq5.4}
\int_b^1\frac{dx}{x(x-a)\sqrt{ (x-b)(1-x)  }   }
=\frac{\pi}{a(1-a)}  
\left(  \frac{a-1}{\sqrt {b}}+\frac{1}{c}  \right)
\end{equation}
holds,
where $c=\sqrt{{(b-a)}/{(1-a)}    } .$

\end{lemma}

\begin{proof}
From a direct computation,
we obtain
\begin{equation}\label{eq5.5}
\frac{d}{du}\left( \frac{a-1}{\sqrt b}\arctan\frac{u}{\sqrt b}+\frac{1}{c}\arctan\frac{u}{c}
\right)=\frac{a-1}{u^2+b}+\frac{1}{u^2+c^2  }  
\end{equation}
and 
\begin{equation}\label{eq5.6}
\frac{du}{dx}=\frac{(1-b)u}{2(x-b)(1-x)},
\end{equation}
where 
$u=\sqrt{{(x-b)}/{(1-x)}   }.$  
    Since  $c^2=(b-a)/(1-a),$ we get
\begin{equation}\label{eq5.7}
\frac{a-1}{u^2+b}+\frac{1}{u^2+c^2}=\frac{a(1-a)(1-x)}{(1-b)x(x-a)}.
\end{equation}
By \eqref{eq5.5}, \eqref{eq5.6} and \eqref{eq5.7},
we have
$$\frac{d}{dx}\left( \frac{a-1}{\sqrt b}\arctan\frac{u}{\sqrt b}+\frac{1}{c}\arctan\frac{u}{c} 
\right)=\frac{a(1-a)}{2}\frac{1}{x(x-a)\sqrt{ (x-b)(1-x)  }}.$$
This implies that
 \begin{equation*}
\int\frac{dx}{x(x-a)\sqrt{ (x-b)(1-x)  }   }=\frac{2}{a(1-a)}\left(\frac{a-1}{\sqrt b}\arctan\frac{u}{\sqrt b}+\frac{1}{c}\arctan\frac{u}{c}
\right)
\end{equation*}
holds.
Hence, we obtain \eqref{eq5.4}.
$\qedd$
\end{proof}

By \eqref{eq5.1} and \eqref{eq5.3},
we get $\inf m_\lambda=1/\sqrt\lambda$ and $m_\lambda'(t)<0$ for any $t>0.$
Hence the half period function $\varphi(\nu)$  for $M_\lambda$ is defined on $(1/\sqrt\lambda,1).$

%%%%%%%%%%%%%%%%%%%%%%%%%%%%%Lemma 2%%%%%%%%%%%%%%%%%%%%%%%%%%%%%%%%%%%%%%%%%%
\begin{lemma}\label{lem5.3}
The half period function $\varphi(\nu)$ is given by
\begin{equation*}
\varphi(\nu)=
{\pi}\left(-\sqrt{\lambda-1}+\frac{\lambda\nu}{\sqrt{\lambda\nu^2-1 } }   \right)
\end{equation*}
on $(\frac{1}{\sqrt\lambda},1)$.
In particular $\varphi$ is decreasing on $(\frac{1}{\sqrt\lambda},1)$
and the surface $M_\lambda$ satisfies the assumptions of the Main Theorem.
\end{lemma}
\begin{proof}
By putting $x:=m^2_\lambda(t),$ we get, by \eqref{eq5.3},
\begin{equation}\label{eq5.8}
dt=\frac{h(t)}{2(1-\lambda)x\tanh t}dx.
\end{equation}
Since $x={(1+\sinh t^2)}/{h(t)},$
\begin{equation}\label{eq5.9}
\sinh^2t=\frac{1-x}{\lambda x-1}, \quad \cosh^2t=\frac{(\lambda-1)x}{\lambda x-1}, \quad {\rm and\:\:} h(t)=\frac{(\lambda-1)}{\lambda x-1}.
\end{equation}
By combining \eqref{eq5.8} and  \eqref{eq5.9},
we obtain,
\begin{equation}\label{eq5.10}
dt=\frac{-\sqrt{\lambda-1}}{2(\lambda x -1)\sqrt{x(1-x)}}dx.
\end{equation}
Therefore, by \eqref{eq1.1},
$$\varphi(\nu)=\nu\sqrt{\lambda-1}\int^1_{\nu^2}\frac{dx}{x(\lambda x-1)\sqrt{(x-\nu^2 )(1-x )} }$$
for $\nu\in(1/\sqrt\lambda,1).$
It follows from Lemma \ref{lem5.2} that 
$
\varphi(\nu)=
{\pi}\left(-\sqrt{\lambda-1}+{\lambda\nu}/{\sqrt{\lambda\nu^2-1 } }   \right).$
It is easy to check that 
$\varphi'(\nu)={-\pi}\left( {1}/{ ( 2\sqrt{\lambda-1})}+{\lambda}/{\sqrt{   \lambda\nu^2-1} ^3   }\right)<0$
on $(1/\sqrt\lambda,1).$
Therefore, by Proposition \ref{prop2.3}, the surface $M_\lambda$ satisfies the assumptions of the Main Theorem.

$\qedd$
\end{proof}

\bigskip

%%%%%%%%%%%%%%%%%%%%%%%%%%%%%%%%%%%%%%%%%%%%%%%%%%%%%%%%%%%%%%%%%%%%%%%%%%%%
{\noindent\large\bf Acknowledgments}

I would like to express my gratitude to Professor Minoru TANAKA 
who kindly gave me guidance for the lectures and  numerous comments.

%%%%%%%%%%%%%%%%%%%%%%%%%%%%%%%%%%%%%%%%%%%%%%%%%%%%%%%%%%%%%%%%%%%%%%%%%%%

\bigskip

\begin{center}
Pakkinee CHITSAKUL\\

\bigskip

Department of Mathematics\\
King Mongkut's Institute of Technology Ladkrabang\\
Ladkrabang, Bangkok\\
10\,--\,520 Thailand

\medskip
{\tt kcpakkin@kmitl.ac.th}
\end{center}

\end{document}